\documentclass[11pt]{amsart}

\usepackage{amssymb}
\usepackage{amsthm}
\usepackage{amsmath}
\usepackage[all]{xy}
\usepackage[margin=1in]{geometry}

\usepackage{ulem}

\usepackage[usenames]{color}
\definecolor{ANDREW}{RGB}{255,127,0}

\usepackage{tikz}
\usetikzlibrary{arrows,calc}

\usepackage{tikz-cd}

\usepackage[foot]{amsaddr}

\usepackage{hyperref}

\theoremstyle{plain}

\newtheorem{thm}     {Theorem}[section]
\newtheorem{prop}    [thm]{Proposition}
\newtheorem{definition}  [thm]{Definition}

\newtheorem{lemma}   [thm]{Lemma}
\newtheorem{example} [thm]{Example}

\newcommand{\B}{\mathbb B}
\newcommand{\C}{\mathbb C}
\newcommand{\D}{\mathbb D}

\newcommand{\R}{\mathbb R}

\begin{document}

\title{On the Kobayashi metrics in Riemannian manifolds}

\author[H. Gaussier]{Herv\'e Gaussier$^1$}
\address{H. Gaussier: Univ. Grenoble Alpes, CNRS, IF, F-38000 Grenoble, France}
\email{herve.gaussier@univ-grenoble-alpes.fr}

\author[A. Sukhov]{Alexandre Sukhov$^2$}
\address{A. Sukhov: University  of Lille,   Laboratoire
Paul Painlev\'e,  Department of 
Mathematics, 59655 Villeneuve d'Ascq, Cedex, France, and
Institut of Mathematics with Computing Centre - Subdivision of the Ufa Research Centre of Russian
Academy of Sciences, 45008, Chernyshevsky Str. 112, Ufa, Russia.}
\email{sukhov@math.univ-lille1.fr}

\date{\today}

\thanks{$^1\,$Partially supported by ERC ALKAGE}
\thanks{$^2\,$Partially suported by Labex CEMPI}


\begin{abstract}
 We define the Kobayashi distance and the Kobayashi-Royden infinitesimal metric on any smooth Riemannian manifold $(M,g)$, using conformal harmonic immersions from the unit disk in $\C$ into $M$. We also study their basic properties, following the approach developped by to H.L.Royden \cite{Ro} for complex manifolds.
 \end{abstract}

\maketitle

\section{Introduction}
The Kobayashi metric and the related notion of hyperbolicity of complex manifolds play a fundamental role in the geometric Complex Analysis and in Algebraic Geometry. The definition of the Kobayashi metric crucially uses holomorphic discs i.e., the holomorphic maps from the unit disc $\D$ of $\C$ to a prescribed complex manifold. Of course, this definition requires the presence of a complex (or, at least, almost complex) structure on a manifold. In \cite{Gr}, M.Gromov proposed to extend the notion of Kobayashi metric to the case of arbitrary Riemannian manifolds. He suggested to use conformal harmonic maps or, equivalently, conformal maps whose images are minimal surfaces, from $\D$ to a Riemannian manifold 
$(M,g)$. Recently, F.Forstneri\v c - D.Kalaj \cite{Fo-Ka} and B.Drinovec-Drnov\v sek - F.Forstneri\v c \cite{Dr-Fo} used this approach. They introduced the notion of Kobayashi metric and studied its important properties in the case of $\R^n$ equipped with the standard Euclidean metric $g_{st}$. The goal of the present paper is to consider the case of arbitrary Riemannian manifolds as it was suggested by M.Gromov. In particular, we extend to that general case some of the results of \cite{Dr-Fo}.

The paper is organized as follows. In Section 2 we recall some standard facts concerning minimal sufaces, conformal and harmonic maps. Lemma \ref{LemNW}, giving an existence of a (small) minimal surface with prescribed center and tangent direction, is necessary for the definition of the Kobayashi pseudodistance and of the Kobayashi-Royden pseudometric. Note that our approach is based on important works of B.White \cite{Wh1,Wh3}.  In Section 3, we introduce the Kobayashi-Royden pseudometric on a Riemannian manifold. The main result here is Theorem \ref{ThSemi} establishing the upper semi-continuity of that pseudometric on the tangent bundle. In Section 4, we introduce the notion of Kobayashi pseudodistance on a Riemannian manifold. The main result of this paper is  Theorem  \ref{KobTh} which claims the coincidence of the Kobayashi pseudodistance with the integral form of the Kobayashi-Royden pseudometric. In the complex case, this is a classical theorem of H.L.Royden  \cite{Ro}. In the last Section 5, we discuss some basic properties of Riemannian manifolds which are hyperbolic in the sense of Kobayashi i.e., the associated Kobayashi pseudodistance is a distance.

Finally, we note that many open questions, such as the sufficient conditions for complete hyperbolicity, the existence of normal families of minimal surfaces, the asymptotic estimates of the Kobayashi-Royden metric near the boundary, or the Gromov hyperbolicity of the Kobayashi distance, stay open in the case of Riemannian manifolds. We will consider them in forthcoming papers. The authors are grateful to F. Forstneri\v c bringing their attention to the papers  \cite{Dr-Fo, Fo-Ka}.

\section{Minimal  immersions}

We denote by $\D$ the unit disc in $\R^2$, by $ds^2$ the standard Riemannian metric on $\R^2$ and by $dm$ the standard Lebesgue measure on $\R^2$. For every $x \in \R^2$ and every $\lambda > 0$, we set $D(x,\lambda):=\{\zeta \in \C/\ |\zeta| < \lambda\}$.

Let $(M,g)$ be a Riemannian manifold. We assume that all structures are smooth of class $C^\infty$. We denote by $dist_g$ the distance induced by $g$, defined as the infimum of the length of $C^1$ paths joining two points. For every $p \in M$ and every $r > 0$, let $\B(p,r):=\{q \in M/\ dist_g(p,q) < r\}$.

A map $u: \D \to M$ is called harmonic if it is a critical point of the energy integral

\begin{eqnarray*}
E(u) = \int_\D \vert du \vert^2dm.
\end{eqnarray*}

 A harmonic map satisfies the Euler-Lagrange  equations 
 
 \begin{eqnarray}
\label{EL}
\Delta u + A(u) (du, du) = 0.
\end{eqnarray}

This is a second order elliptic PDE system. The initial regularity of $u$ may be prescribed in  the H\"older or Sobolev spaces. It follows by the elliptic regularity that $u$ is a smooth $C^{\infty}$ map on $\D$. Using the complex coordinates $z = x + iy$ on $\D$, one can rewrite the equations (\ref{EL}) in the form 

\begin{eqnarray}
\label{EL2}
\frac{\partial^2u}{\partial z \partial \overline z} + \Gamma^{i}_{jk}(u(z)) \frac{\partial u^j}{\partial z}\frac{\partial u^k}{\partial \overline z} = 0
\end{eqnarray}
(see \cite{Jo2}).

A smooth map $u : \D \to M$ is called conformal if  the pull-back $u^*g$ is a metric conformal to $ds^2$ i.e.,
there exists a smooth function $\phi$ such that $u^*g = e^\phi ds^2$ on $\D$. Recall that  any Riemannian metric 
$h$ on $\D$ admits conformal coordinates. This means that there exists a smooth diffeomorphism $\Phi: \D \to \D$, depending on $h$, such that $\Phi^*h$ is conformal to $ds^2$. This is a classical fact of  differential geometry, see \cite{Jo2}. Note that this result is true without additional assumptions only for manifolds of dimension 2.
If $u: \D \to M$ is a smooth immersion, we take $h = u^*g$, and the composition $u \circ \Phi$ becomes a conformal mapping. 
Of course, $\Phi$ depends on $u$ and is not unique. Using the group of conformal automorphisms of $\D$, we can always achieve the conditions 
$\Phi(0) = 0$ and $\Phi(1) = 1$. 

The condition for $u$ to be conformal is equivalent to the conditions, satisfied for every $(x,y) \in \D$:

\begin{eqnarray}
\label{conf}
g_{u(x,y)}(\partial u/\partial x, \partial u/ \partial x) = g_{u(x,y)}(\partial u/\partial y, \partial u/ \partial y), \, \,\, g_{u(x,y)}(\partial u/\partial x, \partial u/ \partial y) = 0
\end{eqnarray}
(see \cite{Jo2}). 

Recall that a surface in $(M,g)$ is called minimal if its mean curvature (induced by $g$) vanishes. A conformal  immersion (i.e., its image) is minimal if and only if 
it is harmonic (see \cite{Jo2}).  The energy functional has very important compactness properties in suitable functional spaces. This makes it a usuful tool in order to study 
boundary values problems for minimal surfaces, in particular, the Plateau problem. 
However, in some cases it is more convenient to work with the area functional. Here we follow the approach of B. White \cite{Wh1}.

Let $u:\D \longrightarrow M$ be a smooth immersion. We denote by $g_{\mathbb D}:=u^*(g)$ the Riemannian metric on $\mathbb D$, pullback of the metric $g$ by $u$. Let $G_{\mathbb D}$ be the matrix $(G_{\mathbb D})_{i,j} = g_{\mathbb D}(\partial/\partial x_i,\partial/\partial x_j)$, for $i, j \in\{1,2\}$ (for convenience, in the matrix notations $x_1=x,\ x_2=y$). We also denote $u_x:=u_{\star}\left(\partial / \partial x\right )$ and $u_y:=u_{\star}\left(\partial / \partial y\right)$. In particular, the scalar product of $\partial / \partial x_i, \partial / \partial x_j$ evaluated at $(x,y) \in \mathbb D$ is equal to $g_{\mathbb D}(\partial / \partial x_i, \partial / \partial x_j) = g_{u(x,y)}(u_{x_i},u_{x_j})$. For convenience, we just write $g_{u(x,y)}(u_{x_i},u_{x_j})=:g_{u}(u_{x_i},u_{x_j})$. Then the area functional $A(u)$ of the immersion $u$ is defined by

\begin{eqnarray}
A(u) = \int_\D (\det G_{\mathbb D})^{1/2}dxdy = \int_{\mathbb D} \sqrt{g_{u}(u_x,u_x) g_{u}(u_y,u_y)-g_{u}(u_x,u_y)g_{u}(u_x,u_y)} \; dx dy.
\end{eqnarray}

One may view $A$ as a real map defined on the space of smooth immersions. A smooth immersion $u$ is called {\bf stationary} if the differential $DA$ of $A$ vanishes at $u$ i.e., $DA(u) = 0$. As it is shown in \cite{Wh1}, an immersion is stationary if and only if its image is a minimal surface. Therefore, after a suitable reparametrization, a stationary immersion  becomes a conformal harmonic map.

In the rest of the paper, we refer conformal harmonic immersions from $\D$ to $M$ as conformal harmonic immersed discs. Similarly, we  refer to stationary or minimal disc. Note also that we sometimes identify an immersed disc with its image when this does not lead to any confusion.



Since we work only locally, we will consider $u : \mathbb D \rightarrow \mathbb R^{n}$ given as a graph :
$$
\forall (x,y) \in \mathbb D, \ u(x,y) = (x,y,u^3(x,y), \dots, u^n(x,y)),
$$
where $u^3, \dots, u^n$ are smooth $C^{\infty}$ functions. Then
$$
u_x:=(1,0,u^3_x, \dots, u^n_x)
$$
and
$$
u_y:=(1,0,u^3_y, \dots, u^n_y).
$$

Now we follow White's approach in \cite{Wh1, Wh3}. Let $h=(0,0,h^3, \dots, h^n):\D \rightarrow \R^{n}$ be a smooth map and let, for every $t \in [0,1]$ and every $(x,y) \in \D$, $\varphi^t(x,y) = (x,y,u^3(x,y)+th^3(x,y),\dots, u^n(x,y)+th^n(x,y))$. Then
$$
\varphi^t_x = (\varphi^t)_{\star}(\partial / \partial x) = (1,0,u^3_x+th^3_x, \dots, u^n_x + t h^n_x),
$$

$$
\varphi^t_y = (\varphi^t)_{\star}(\partial / \partial y) = (1,0,u^3_y+th^3_y, \dots, u^n_y + t h^n_y).
$$

and

$$
A(\varphi^t) = \int_{\mathbb D} \sqrt{g_{\varphi^t}(\varphi^t_x,\varphi^t_x) g_{\varphi^t}(\varphi^t_y,\varphi^t_y)-g_{\varphi^t}(\varphi^t_x,\varphi^t_y)g_{\varphi^t}(\varphi^t_x,\varphi^t_y)}dx dy.
$$

If we write $h_x=(0,0,h^3_x, \dots, h^n_x)$ and $h_y=(0,0,h^3_y, \dots, h^n_y)$, then we have, for every $t \in [0,1]$:
$$
\varphi^0_x = u_x, \ \ \ \frac{d}{dt}_{|t=0} \varphi^t_x = h_x
$$
and
$$
\varphi^0_y = u_y, \ \ \ \frac{d}{dt}_{|t=0} \varphi^t_y = h_y.
$$

Then $u$ is a stationary immersion if and only if for every smooth map $h$ we have
$$
\frac{d}{dt}_{|t=0} A(\varphi^t) = 0.
$$


The expression of  $\frac{d}{dt}_{|t=0} A(\varphi^t)$ is straightforward and depends linearly on $h$, $h_x$ and $h_y$. By integration by parts, we obtain a quasilinear operator, with respect to $u_x$, $u_y$, $u_{xx}$ and $u_{yy}$, that depends linearly on $h$. Using the Riesz representation theorem (see details in \cite{Wh1}), we obtain the following

\begin{lemma}\label{min-lemma}
With these notations, the stationary condition has the form

\begin{eqnarray}
\label{stat1} 
H(u) = 0
\end{eqnarray}
where
\begin{eqnarray}
\label{stat2}
H(u) = (H^1(u),...,H^n(u))
\end{eqnarray}
with
$$
H^j(u) = \psi_{j}(u_x,u_y,u_{xx},u_{xy},u_{yy}) \ {\rm for \ } j=1,2
$$
and
$$
H^j(u) = u^j_{xx}+u^j_{yy}+\psi_j(u_x,u_y,u_{xx},u_{xy},u_{yy}) \ {\rm for \ } j = 3,\dots, n.
$$

Here, for $j=1, \dots, n$, $\psi_j$ is a smooth $C^{\infty}$ function, without constant or linear terms with respect to $u_x,u_y,u_{xx},u_{xy},u_{yy}$. Furthermore, the vector $H(u)$ is orthogonal to $u(\D)$.

\vspace{1mm}
In particular, the operator $H$ is a quasilinear elliptic operator whose linearization at $u: (x,y) \in \D \mapsto (x,y,0,\dots,0) \in \R^n$ is

\begin{eqnarray}\label{linear-eq}
\left(0, 0, \Delta u^3, \cdots, \Delta u^n\right)
\end{eqnarray}

where $\Delta$ denotes the standard Laplace operator.
\end{lemma}

Since $H(u)$ is orthogonal to $u(\D)$, the equation (\ref{stat1}) is equivalent to the equations
\begin{eqnarray}
\label{stat11}
 H^j(u) = 0, \ j = 3,...,n.
 \end{eqnarray}

  In what follows we mean these equations when refering to (\ref{stat1}).

\vspace{2mm}
\begin{example}
Consider the special case where $M = \R^3$ and $g$ is the standard metric. Assume that the stationary map $u:\D \rightarrow \R^3$ is the graph of a function $f:\D \to \R$ i.e., $u:(x,y) \in \D \mapsto (x,y,f(x,y))$. Then the equation (\ref{stat11}) takes the form

\begin{eqnarray}
\label{standard}
(1+ u_y^2)u_{xx}  + (1 + u_x^2)u_{yy} - 2 u_xu_y u_{xy} = 0.
\end{eqnarray}
Its linearization at $u : (x,y) \in \D \mapsto (x,y,0) \in \R^3$ is the standard Laplace operator $\Delta$.
\end{example}

We need the following classical property of this operator. Consider the logarithmic potential

\begin{eqnarray}
K(z) = (2\pi)^{-1} \log \vert z \vert.
\end{eqnarray}
This is a fundamental solution to the Laplace equation $\Delta u = 0$.  Consider the Poisson equation

\begin{eqnarray}
\label{Pois}
\Delta u = f \,.
\end{eqnarray}
Assume that $f$ belongs to the H\"older class $C^\mu(\overline \D)$, where $\mu$ is a positive real number. It is classical (see \cite{Mo1}) that the potential of $f$ defined by
\begin{eqnarray}
Uf(z) = \int_{\D} K(z - w) f(w) dm(w)
\end{eqnarray}
is a function of class $C^{2 + \mu}(\overline \D)$ and satisfies $\Delta U = f$. Furthermore, the linear operator 
$U: f \mapsto Uf$, $U: C^\mu(\overline\D) \to C^{2 + \mu}(\overline \D)$, is bounded.

We denote $e_1:=(1,0) \in \R^2$. The following result claims that locally there are many conformal harmonic maps.

\begin{lemma}
\label{LemNW}
\begin{itemize}
\item[(i)]  Let $p \in M$ and $\xi \in T_pM \setminus \{ 0 \}$. Then there exists a conformal harmonic immersion 
$u: \D \to M$ such that $u(0) = p$ and $du(0) \cdot e_1 = \alpha \xi$ for some $\alpha > 0$. Furthermore, this immersion depends smoothly on  $p$ and $\xi$.
\item[(ii)] Let $u: \D \to M$ be a conformal harmonic immersion. Suppose that $u(\D)$ is contained in a ball $\B(p,r)$ of radius $r > 0$ small enough.
Then there exists $r_1 < r$ and a smooth $(n-2)$-parameter foliation of $B(p,r_1)$ by conformal harmonic discs containing $u$ as a leaf. Moreover, given $q$ 
close enough to $p$, and a vector $\xi$ close enough to $du(0) \cdot e_1$, one can choose the above foliation such that for some leaf $\tilde u$ one has $\tilde u(0) = q$ and $d\tilde u (0) \cdot e_1 = \xi$.
\end{itemize}
\end{lemma}
\begin{proof} (i) We consider local coordinates $(x_1,\dots,x_n)$ on $M$ in which $p = 0$. We also assume that 
these coordinates are normal for the Levi - Civita connexion of $(M,g)$. Then in these coordinates $g_{ij}(0) = \delta_{ij}$
and the first order partial derivatives of $g_{ij}$ vanish at $0$. Consider the metric $g_t(x_1,\dots,x_n) := g(tx_1, \dots, tx_n)$ for $t > 0$. We notice that such a metric is isometric to $g$.
The metric $h_t = t^{-2}g_t$ is not isometric to $g_t$, but the corresponding  set of stationary surfaces is the same as for $g_t$ ; this follows immediately from the expression for the area functional and the definition of stationary immersions. Finally, note that the metric $h_t$ converges to the standard metric $g_{st}$ of $\R^n$ in any $C^k$-norm on any compact subset of  $\R^n$, as $t \rightarrow 0$.



We may assume that $\xi = (1,0,...,0)$ and we search for a suitable immersion of the unit disc of the form $(x,y) \mapsto (x,y,f(x,y))$ i.e., as the graph of a vector function $f: \D \to \R^{n-2}$. Consider the equation (\ref{stat11}) for 
the metric $h_t$. For $t = 0$, it becomes a vector analog of (\ref{standard}). Its linearization at $f=0$ is given by (\ref{linear-eq}) which is a surjective operator  from $C^{2,\alpha}(\D, \R^{n-2})$ to $C^{0,\alpha}(\D,\mathbb R^{n-2})$. By the implicit function theorem the equation~(\ref{stat1}) admits solutions for $t > 0$ small enough. Namely, for every sufficiently small $t$, there exists a minimal surface, given by a smooth conformal harmonic immersion $u_{t,p,\xi} : \D \rightarrow \mathbb R^n$, for the Riemannian metric $h_t$, such that $u_{t,p,\xi}(0)$ is close to $p$ and $du_{t,p,\xi}(0) \cdot e_1$ is close to $\xi$. Now, if $\mathcal U_{p}$ is a small neighborhood of $p$ in $\R^n$ and $\mathcal V_{\xi}$ is a small neighborhood of $\xi$ in the unit sphere (for the standard metric in $\R^n$), the same reasoning implies that for every sufficiently small $t$, the set $\{(u_{t,p',\xi'}(0), du_{t,p',\xi'}(0) \cdot e_1), \ p' \in \mathcal U_{p},\ \xi' \in \mathcal V_{\xi}\}$ fills an open neighborhood of $(p,\xi)$ in $\R^n \times \R^n$. Hence, for sufficiently small $t$, there exists $(p',\xi') \in \mathcal U_{p} \times \mathcal V_{\xi}$ such that $u_{t,p',\xi'} (0) = p$ and $du_{t,p',\xi'} (0) \cdot e_1 = c \xi$ for some real number $c$ close to one.  
Finally, being a solution of the equation~(\ref{stat1}), $u_{t,p',\xi'} : \D \rightarrow \R^n$ is a conformal harmonic immersion  for $g_t$ (after a suitable reparametrization), and therefore for $g$. The smooth dependence on $p$ and $\xi$ follows from the implicit function theorem. This proves Part (i) of Lemma~\ref{LemNW} . 

(ii) Choose local coordinates $(x_1, \dots, x_n)$ as above. Also, set $(\tilde x_1, \dots, \tilde x_n) = (tx_1, \dots, tx_n)$ for $t > 0$ small enough and again consider the metrics $h_t(\tilde x_1, \dots, \tilde x_n) = t^{-2}g(tx_1, \dots, t x_n)$. We may assume that, in the coordinates $(x_1,\dots,x_n)$, the conformal harmonic immersion $u:\D \rightarrow M$ has the form  $u(x,y) = L(x,y) + O(\vert (x,y) \vert^2)$ where $L$ is a linear map from $\R^2$ to $\R^n$, of rank 2. Then the disc $u_t:(x,y)\in \D \mapsto u(tx, ty)$ has the expansion $u_t(x,y) = L(x,y) + O(t)$  in the coordinates $(\tilde x_1, \dots, \tilde x_n)$. As $t \to 0$, this family of discs converges to a conformal linear disc. Note that after an isometric (with respect 
to  $g_{st}$ structure ) this disc is a graph of the zero map. Then the desired result follows  by the implicit function theorem as in the part (i). 


\end{proof}

\bigskip

Note that there is another way to prove Lemma~\ref{LemNW}, based on deep results on the Plateau problem.
The following result is a very special case of the fundamental theorem of Morrey \cite{Mo2}. Let $\phi: \overline \D \to (M,g)$ be a smooth map. Suppose that a Riemannian manifold $(M,g)$ satisfies some standard metric assumptions (for example, $M$ is compact). Then there exists a smooth map $u: \overline \D \to M$ such that $u \vert_{b\D} = \phi \vert_{b\D}$. 
Such a map is conformal and harmonic. In particular, the image of $u$ is a minimal surface in $M$ with the boundary $u(b\D)$. Furthermore, it follows from the results of Hildebrandt
\cite{Hi} that $u$ depends continuously on the perturbation of $\phi$. Choose local coordinates on $M$ as above and consider all conformal linear maps from $\D$ to this local chart. Then the theorem of Morrey can be applied to the images of the unit circle by these maps. Since, in these local coordinates, the metric $g$ is a small perturbation of the standard one, minimal surfaces given by Morrey's theorem are small deformations of the linear discs. This implies Lemma~\ref{LemNW}.

\section{The Kobayashi-Royden pseudometric on a Riemannian manifold}

Let $(M,g)$ be a Riemannian manifold. For a point $p \in M$ and a tangent vector   $\xi \in T_pM$ we set
$$F_M(p,\xi) := \inf \frac{1}{r}$$
where $r$ runs over all positive real numbers for which there exists a conformal harmonic immersion  $u: \D \to M$ such that 
$u(0) = p$ and $du(0) \cdot e_1 = r \xi$. It follows by Lemma \ref{LemNW} that $F_M$ is well defined for every $(p,\xi)$ in the tangent bundle $TM$. We call $F_M$ the {\sl Kobayashi-Royden pseudometric} for the Riemannian manifold 
$(M,g)$.

\begin{lemma}
\label{Lem1}
Let $f: (M,g) \to (N,h)$ be an isometry  between two Riemannian manifolds i.e., $f$ is differentiable and satisfies $g = f^*h$. Then 
$$F_N(f(p),df(p)\xi) \le F_M(p,\xi).$$
In particular, if $M$ is a connected open subset of $N$, then 
$$F_N(p,\xi) \le F_M(p,\xi).$$
\end{lemma}

For $(p,\xi) \in TM$, we denote $\vert \xi \vert_g:=\sqrt{g_p(\xi,\xi)}$, when no confusion is possible.
\begin{lemma}
\label{Lem2}
The function $F_M$ is non-negative, and for any real $a$ one has
$$F_M(p, a\xi) = \vert a \vert F_M(p,\xi).$$
If $K$ is a compact subset of $M$  then there is a constant $C_K > 0$ such that 
$$F_M(p,\xi) \le C_K \vert \xi \vert_g.$$
\end{lemma}
\begin{proof} For every point $p \in M$ and every sufficiently small open neighborhood $U$ of $p$, Lemma \ref{LemNW} implies that there exists $\varepsilon > 0$ with the following property: for every point $q \in U$, every unit vector $\xi \in T_qM$ and every $r \in (0,\varepsilon)$, there exists a conformal harmonic immersion $u: \D \to M$ such that $u(0) = q$, $du(0) \cdot e_1 = r \xi$.  Then for $C _U= 1/\varepsilon$ we have 
$F(q,\xi) \le C_U \vert \xi \vert_g$ for all vectors $\xi$ (not nesessarily unit). Covering $K$ by a finite number of open neighborhoods, we conclude.
\end{proof}

The main regularity property is given by the following

\begin{thm}
\label{ThSemi}
The function $F_M$ is  upper semi-continuous on the tangent bundle $TM$.
\end{thm}
The key result needed for the proof is the following
\begin{prop}
\label{PropSemi}
Let $u_0:\D \to M$ be a conformal harmonic immersion. Then there exists a neighborhood $U$ of 
$(u(0),du_0(0) \cdot e_1)$ in the tangent bundle $TM$ such that for each $(p,\xi) \in U$ there exists a conformal harmonic immersion 
$u:\D \to M$ satisfying $u(0) = p$, $du(0) \cdot e_1 = \xi$.
\end{prop}
The upper semicontinuity of $F_M$ follows immediately from this proposition and the definition of $F_M$. So the  
remainder of this section is devoted to the proof of Proposition \ref{PropSemi}.

\begin{proof}
Essentially, the proof is implicitely contained in the works
\cite{Wh1,Wh3}. For the convenience of the reader we include some details. Without loss of generality we identify $M$ with $\R^n$ using the  Nash isometric embedding theorem.


The operator $DH(u_0)$ is called the Jacobi operator at $u_0$. In order words, the Jacobi operator is the linearization of $H$ at $u_0$. This is a second order linear PDE operator.
Furthermore, it is elliptic and self-adjoint. If a vector field $h$ is tangent to the disc $u_0(\D)$, then the value of $DH(u_0)$ on $h$, i.e., $DH(u_0) \cdot h$, vanishes. Thus, for any  $h$, 
the vector field $DH(u_0) \cdot h$ is  normal to $u_0(\D)$. Similarly to Section 2 (cf. the equations (\ref{stat1}) and (\ref{stat11})),  in what follows we identify $DH(u_0)$ with its normal projection  (with respect to $u_0(\D)$). A Jacobi field with respect to $u_0$ is a vector field in the kernel of $DH(u_0)$. It is established in \cite{Wh1} (Corollary  at Section 6) that a smooth Jacobi field at $u_0$ is the initial velocity vector field of a one-parameter family of stationary discs. In order to apply this result, we need to study the existence of Jacobi fields. 

The existence of the Jacobi fields in the direction transverse to $u_0(\D)$ is established in \cite{Wh3}. Consider a sufficiently small $r > 0$ and the disc $u_0(r\D)$. By Lemma \ref{LemNW}
this disc generates a one-parameter family of minimal discs in any direction tranverse to $u_0(\D)$ at the origin. Such a deformation is generated by a Jacobi field, obtained by taking the derivative with respect to the parameter, defined in a neighborhood of the origin on $u_0(\D)$: this provides the existence of a Jacobi field with respect to $u_0$  in a prescribed direction near the centre $u_0(0)$. The existence of a global Jacobi field now follows from P.Lax' Equivalence Theorem  of \cite{La},
stating, in our case, that a Jacobi field defined in a neighborhood of $0$ in $\R^2$ can be approximated by a Jacobi field defined on the whole disc $\D$. The property of  unique continuation of solutions  required by theorem of P.Lax follows in our case from 
the Calderon uniqueness theorem, see \cite{Ca}, Theorem 11. 
Now Proposition \ref{PropSemi} follows.
\end{proof}

\section{The Kobayashi distance and coincidence theorem}

Denote by $P_{\D}$ the Poincar\'e metric, defined for $z \in \D$ and $v \in \C$ by

$$
P_{\D}(z,v) = \frac{\vert v \vert}{1- \vert z \vert^2}
$$

and by $\rho_{\D}$ the Poincar\'e distance on $\D$. Recall that for every $z,w \in \D$ :

$$
\rho_{\D}(z,w) = \frac{1}{2} \log \frac{1 + d(z,w)}{1- d(z,w)}
$$
where
$$
d(z,w) = \frac{\vert z - w \vert}{\vert 1 - z \overline w \vert}.
$$

 Let $p$ and $q$ be two points in the Riemannian manifold $(M,g)$. A Kobayashi chain from $p$ to $q$ is a finite
sequence of points $z_k$, $w_k$ in $\D$ and of conformal harmonic immersions $u_k: \D \to M$, $k = 1, \dots,m,$ such that $u_1(z_1) = p$, $u_m(w_m) = q$ and 
$u_k(w_k) = u_{k+1}(z_{k+1})$ for $k=1, \dots, m-1$. The existence of such chains follows from Lemma \ref{LemNW}. Indeed, by compactness any smooth path from $p$ to $q$ can be covered 
by sufficiently small balls such that the centre of a given ball is contained in the preceding ball. By Lemma \ref{LemNW} each ball is foliated by minimal discs through its centre, which provides a Kobayashi chain. The Kobayashi pseudodistance from $p$ to $q$ is then defined by
\begin{eqnarray}
\label{KobDist1}
d_M(p,q) = \inf \sum_k \rho_{\D}(z_k,w_k)
\end{eqnarray}
where the infimum is taken over all Kobayashi chains from $p$ to $q$.

On another hand, we consider the pseudodistance defined as the integrated form of $F_M$:

\begin{eqnarray}
\label{KobDist2}
\overline d_M(p,q) = \inf \int_0^1 F_M(\gamma(t), \dot{\gamma}(t))dt
\end{eqnarray}
where the infimum is taken over all piecewise smooth paths $\gamma$ from $p$ to $q$. Notice that $\overline{d}_M$ is well defined by Theorem~\ref{ThSemi}. The main result of this section is the following coincidence theorem.

\begin{thm}
\label{KobTh}
We have $d_M = \overline d_M$.
\end{thm}

\begin{proof} First we note that for every conformal harmonic disc $u:\D \to M$ we have, for every $z \in \D$ and every $\xi \in \C$:

\begin{eqnarray}
\label{KobDist3}
F_M(u(z),du(z)\cdot \xi) \le P_{\D}(z,\xi).
\end{eqnarray}
Indeed, in the case where $z=0$ and $\xi = e_1$, this follows from the definition of $F_M$. Then it follows for any $z \in \D$ and any unit vector $\xi$, if we replace $u$ with the conformal harmonic disc $u \circ \phi$, where $\phi$ is a biholomorphic automorphism of $\D$ such that $\phi(0) = z$, $d\phi(0) \cdot e_1 = \xi$. Then the inequality~(\ref{KobDist3}) follows for any vector $\xi$ since the two metrics are absolutely homogeneous. As a consequence, for any conformal harmonic disc $u: \D \to M$ we have for every $z, w \in \D$:
\begin{eqnarray}
\label{KobDist4}
\overline{d}_M(u(z), u(w)) \le \rho_{\D}(z,w).
\end{eqnarray}

Consider now a Kobayashi chain between $p$ and $q$, as above. Then, by the triangle inequality and the inequality~(\ref{KobDist4}) we have:
$$
d_M(p,q) \le \sum_k d_M(u_k(z_k), u_k(w_k)) \le \sum \rho_{\D}(z_k,w_k).
$$

Taking the infimum over all Kobayashi chains joining $p$ to $q$, we obtain that $\overline d_M(p,q) \le d_M(p,q)$. 

Let us prove now the converse inequality. Fix $\varepsilon > 0$ and consider a smooth path $\gamma: [0,1] \to M$ satisfying 
$\gamma(0) = p$, $\gamma(1) = q$ and such that 

$$
\int_0^1 F_M(\gamma(t),\dot{\gamma}(t))dt < \overline d_M(p,q) + \varepsilon.
$$

Since $F_M$ is an upper-semicontinuous function, there is a continuous function $h:[0,1] \to \R$ satisfying 
$h(t) > F_M(\gamma(t),\dot{\gamma}(t))$ for every $t \in [0,1]$ and such that 

$$
\int_0^1 h(t)dt < \overline d_M(p,q) + \varepsilon.
$$

Therefore, for every sufficiently fine partition $0 = t_0 < t_1 <...< t_m = 1$ of $[0,1]$  we have

$$\sum_{i=1}^m h(t_{i-1})(t_i-t_{i-1}) < \overline d_M(p,q) + \varepsilon.$$

\begin{lemma}
\label{LipLemma}
There exist constants $C > 0$ and $\delta > 0$ such that for every $t \in [0,1]$ and every $a,b \in \B(\gamma(t),\delta)$ we have:
$$d_M(a,b) \le C dist_g(a,b)$$
where $dist_g$ denotes the distance induced by the Riemannian metric $g$ on $M$.
\end{lemma}
\begin{proof} By Lemma \ref{LemNW} there exists $\delta > 0$ small enough such that for any $t$ the ball 
$\B(\gamma(t),\delta)$ is foliated by immersed minimal discs, whose centers coincide with $\gamma(t)$. It also follows by the Lemma~\ref{LemNW} 
that for any $a$, $b$ in $\B(\gamma(t),\delta)$ there exists a conformal harmonic immersion $u:\D \rightarrow M$  through $a$ and $b$.
The map $u$ being an immersion, there exists $C > 0$, only depending on $\delta$, such that
$$dist_g(a,b) = dist_g(u(z),u(w)) \geq C |z-w|.
$$
Moreover, by changing $\delta$ is necessary, there exists $C' > 0$, only depending on $\delta$, such that $\rho_{\D}(z,w) \le C' \vert w-z \vert$.

Now it follows from the definition of the Kobayashi distance that

$$
d_M(a,b) \le \rho_{\D}(z,w) \le C \vert w-z \vert \le C \vert b - a \vert.
$$
(Here, the constant $C > 0$ changes from inequality to inequality.)

All constants being uniform, this proves Lemma~\ref{LipLemma}.
\end{proof}
Fix $t \in [0,1]$. Since $h(t) > F_M(\gamma(t), \dot{\gamma}(t))$, there exists a conformal harmonic immersion $u : \D \rightarrow M$ and a real number $r > 0$ such that 
$u(0) = \gamma(t)$, $du(0) \cdot e_1 = r\dot\gamma(t)$ and $h(t) > 1/r$. Therefore, for every $s$ real close enough to the origin we have the expansion, in local coordinates in $M$:
$$
u(s/r) = \gamma(t) + (s/r)du(0) \cdot e_1 + O(s^2) = \gamma(t) + s\dot\gamma(t) + O(s^2).
$$
Note also that $\rho_{\D}(0,z) = \vert z \vert + O(\vert z \vert^2)$ for $z \in \D$ close enough to $0$.
Therefore, using Lemma \ref{LipLemma} we obtain:

\begin{eqnarray*}
d_M(\gamma(t), \gamma(t) + s\dot\gamma(t)) & \le & d_M(u(0), u(s/r)) + d_M(u(s/r), \gamma(t) + s\dot\gamma(t)) \\
& \le & \rho_{\D}(0, s/r) + C dist_g(u(s/r) , \gamma(t) + s\dot\gamma(t)) \\
& \le & \vert s \vert/r + O(s^2) \\
& < & \vert s \vert h(t) + O(s^2).
\end{eqnarray*}
Using again Lemma \ref{LipLemma}, we conclude that for $t, \tilde t \in [0,1]$, close enough, one has

\begin{eqnarray*}
d_M(\gamma(t), \gamma(\tilde t)) & \le & d_M(\gamma(t), \gamma(t) + (\tilde t - t)\dot\gamma(t)) + d_M(\gamma(t) + (\tilde t - t)\dot\gamma(t),\gamma(\tilde t)) \\
& \le & \vert \tilde t - t \vert h(t)  + O(\vert \tilde t - t \vert^2).
\end{eqnarray*}
As a consequence, for any sufficiently fine partition we have
\begin{eqnarray*}
d_M(p,q) \le \sum_{i=1}^m d_M(\gamma(t_{i-1}), \gamma(t_i)) \le \sum_{i=1}^m (t_i - t_{i-1}) h(t_{i-1})(1 + \varepsilon) < (1+ \varepsilon)(\overline d_M(p,q) + \varepsilon).
\end{eqnarray*}
Since $\varepsilon > 0$ is arbitrary, the proof is concluded.

\end{proof}

\section{Hyperbolicity}
The main purpose of this Section is to prove an analogue of Theorem~2, p. 133, of \cite{Ro}. This is proved in the context of Riemannian manifolds, when $(M,g) = (\mathbb R^n,g_{st})$, where $g_{st}$ denotes the standard Euclidian metric, in \cite{Dr-Fo}, Theorem 4.2. We follow the presentation of \cite{Ro}.

\begin{definition}\label{hyp-def}
\begin{itemize}
\item[$(i)$] A Riemannian manifold $(M,g)$ is called {\sl hyperbolic} at a point $x \in M$ if there is a neighborhood $U$ of $x$ and a positive constant $c$ such that $F_M(y,\xi) \geq \vert \xi\vert_g$ for every $y \in U$ and every $\xi \in T_yM$.

\item[$(ii)$] A Riemannian manifold $(M,g)$ is called {\sl tight} if the family of conformal harmonic immersions from $\D$ to $M$ is equicontinuous for the topology generated by $dist_g$.

\item[$(iii)$] A family $\mathcal F$ of mappings of a topological space $X$ into a topological space $Y$ is called  {\sl even} if, given $x \in X,\ y \in Y$ and a neighborhood $U$ of $y$, there is a neighborhood $V$ of $x$ and a neighborhood $W$ of $y$ such that for every $f \in F$, we have $f_{|V} \subset U$ whenever $f(x) \in W$.

\end{itemize}
\end{definition}

We have the following characterization of hyperbolicity of Riemannian manifolds
\begin{thm}
\label{HypTh}
Let $(M,g)$ be a Riemannian manifold. Then the following statements are equivalent:

\begin{itemize}
\item[$(i)$] the family $\mathcal C \mathcal H(\D,M)$ of conformal harmonic immersions from $\D$ to $M$ is equicontinuous with respect to the distance $dist_g$ i.e., $(M ,dist_g)$ is tight,

\item[$(ii)$] the family $\mathcal C \mathcal H(\D,M)$ is an even family,

\item[$(iii)$] $M$ is hyperbolic,

\item[$(iv)$] $d_M$ is a metric i.e., $M$ is Kobayashi hyperbolic,

\item[$(v)$] the Kobayashi metric $d_M$ induces the usual topology of $M$.
\end{itemize}
\end{thm}

\begin{proof}
The implication $(i) \Rightarrow (ii)$ follows from \cite{Ke} p. 237, where it is stated that equicontinuity of a family of mappings of a topological space $X$ into a topological space $Y$ with respect to any metric inducing the topology of $Y$ implies that the family is an even family.

\vspace{2mm}
$(ii) \Rightarrow (iii)$. This is based on the following result, see \cite{Jo1}, Theorem 4.8.1 p. 113 :

\vspace{1mm}
\noindent{\bf Theorem (Schwarz Lemma).}
{\sl Let $X$ and $Y$ be Riemannian manifolds. Let $\B(x_0,R_0) \subset X$, $R_0 < \min\left(i_X(x_0),\frac{\pi}{2\kappa_X}\right)$, where $i_X(x_0)$ denotes the injectivity radius at $x_0$ and $-\omega_X^2 \leq K_X \leq \kappa_X^2$ are curvature bounds on $\B(x_0,R_0)$. Let $\B(y_0,R') \subset Y$, $R' < \min\left(i_Y(y_0),\frac{\pi}{2\kappa_Y}\right)$, where $i_Y(y_0)$ denotes the injectivity radius at $x_0$ and $-\omega_Y^2 \leq K_Y \leq \kappa_Y^2$ are curvature bounds on $\B(y_0,R')$.

\vspace{1mm}
If $u : X \rightarrow \B(y_0,R')$ is harmonic, then for all $R \leq R_0$:
\begin{equation}\label{harm-ineq}
|\nabla u(x_0)| \leq c_0 \max_{x \in B(x_0,R)}\frac{d(u(x),u(x_0))}{R}
\end{equation}
where $c_0 = c_0(R_0,\omega_X, \kappa_X, \dim X, R', \omega_Y, \kappa_Y, \dim Y)$.
}

\vspace{1mm}
Let $p$ be a point in $M$ and let $\B(p,R')$ such that $R'$ satisfies the condition of the above Schwarz Lemma, with $y_0 = p$. By assumption, the family $\mathcal C \mathcal H(\D,M)$ being even, there exist $0 < \delta < 1$ and $0 < \delta' < R'$, such that for every conformal harmonic immersion $u: \D \rightarrow M$, we have $u_{|D(0,\delta)} \subset B(p,R')$, whenever $u(0) \in B(p,\delta')$.
It follows by the Schwarz Lemma that there exists $c_0 > 0$ such that for every conformal harmonic immersion $u:\D \rightarrow M$, with $u(0) \in \B(p,\delta')$, we have
$$
|\nabla u(0)| \leq c_0 \frac{R'}{\delta}.
$$
Hence, for every $y \in\B(p,\delta')$ and for every $v \in T_yM$, we have: $\displaystyle F_M(y,v) \geq \frac{\delta}{c_0 R'}\vert v\vert_g$.

\vspace{2mm}
$(iii) \Rightarrow (iv)$. By the assumption on $F_M$ and the definition of $\overline{d}_M$, $\overline{d}_M$ is a distance. The implication is then given by Theorem~\ref{KobTh}.

\vspace{2mm}
$(iv) \Rightarrow (v)$. Let $x \in M$ and let $\delta > 0$ be such that $\B(x,\delta)$ is geodesically convex i.e., for every $y,z \in \B(x,\delta)$, there exists $\gamma_0: [0,1] \rightarrow M$ a piecewise $C^1$ smooth path joining $y$ to $z$, with  $\gamma_0([0,1]) \subset \B(x,\delta)$ and $l_g(\gamma_0) := \int_0^1 |\dot{\gamma_0}(t)|_g dt = dist_g(y,z)$. Applying Lemma~\ref{Lem2} to the compact set $K:=\overline{\B(x,\delta)}$, there exists $C_K > 0$, such that $F_M(y,v) \leq C_K \vert v\vert_g$, for every $y \in K$ and every $v \in T_y M$. It follows now from Theorem~\ref{KobTh}:

\begin{equation}\label{lip-eq}
d_M(y,z) = \overline{d}_M(y,z) \leq \int_0^1 F_M(\gamma_0(t), \dot{\gamma_0}(t))dt \leq C_K \int_0^1 \vert \dot{\gamma_0}(t)\vert_g dt = C_K dist_g(y,z).
\end{equation}
This means that the topology generated by $d_M$ is weaker than the topology generated by $dist_g$.

\vspace{1mm}
Moreover, it follows from (\ref{lip-eq}) that the map $y \mapsto d_M(x,y)$ is continuous on $\B(x,\delta)$ for the topology generated by $dist_g$. In particular, for every $0 < \alpha < \delta$, there exists $y_0 \in \partial \B(x,\alpha/2)$ such that
$$
d_M(x,y_0) = \inf_{y \in \partial \B(x,\alpha/2)}d(x,y) > 0.
$$
By the definition of $\overline{d}_M$, for every $y \in M \setminus \overline{\B(x,\alpha)}$
$$
d_M(x,y) = \overline{d}_M(x,y) \geq \overline{d}_M(x,y_0) = d_M(x,y_0).
$$
This means that the set $\{y \in M / d_M(x,y) < d_M(x,y_0)\} \subset \B(x,\alpha)$ i.e., the topology generated by $dist_g$ is weaker than the topology generated by $d_M$.

\vspace{2mm}
$(v) \Rightarrow (i)$. It follows from (\ref{KobDist1}) et (\ref{KobDist2}) that every conformal harmonic immersion $f:\D \rightarrow M$ and for every $\zeta, \zeta' \in \D$ : $d_M(f(\zeta), f(\zeta'))  \leq \rho_{\D}(\zeta,\zeta')$. Hence the family $\mathcal C \mathcal H(\D,M)$  is equicontinuous.
\end{proof}

We conclude by noticing that, as in the complex setting and following \cite{Ro}, we can define the notions of tautness and complete hyperbolicity for Riemannian manifolds :

\begin{definition}
\begin{itemize}
\item[(i)] A Riemannian manifold $(M,g)$ is called {\sl taut} if the family $\mathcal C \mathcal H(\D,M)$ is a normal family i.e., every sequence in $\mathcal C \mathcal H(\D,M)$ admits a subsequence that either converges to an element of $\mathcal C \mathcal H(\D,M)$ uniformly on compact subsets of $\D$, or is compactly divergent.

\item[(ii)] A Riemannian manifold $(M,g)$ is called {\sl complete hyperbolic} if it is Kobayashi hyperbolic and the metric space $(M,d_M)$ is complete.
\end{itemize}
\end{definition}

Then the characterization of tautness and complete hyperbolicity contained in Section 4 of \cite{Ro} still hold in our context. In particular we have the following

\begin{prop}\label{charact-prop}
Let $(M,g)$ be a Riemannian manifold. Then we have:

\begin{itemize}
\item[(i)] $M$ is complete hyperbolic if and only if for every $p \in M$ and every $r > 0$, the set $\{q \in M / \ d_M(p,q) \leq r\}$ is compact in $M$.

\item[(ii)] $M$ complete hyperbolic $\Rightarrow$ $M$ taut $\Rightarrow$ $M$ Kobayashi hyperbolic.

\item[(ii)] Let $\pi : \tilde M \rightarrow M$ be a covering map of $M$ and denote by $\tilde g$ the unique Riemannian metric on $\tilde M$ such that $\pi$ is a Riemannian covering map. Then $(\tilde M, \tilde g)$ is complete hyperbolic if and only if $(M,g)$ is complete hyperbolic.
\end{itemize}
\end{prop}

{\footnotesize


\begin{thebibliography}{CIT}



\bibitem{Ca} {\sc Calderon, A.P.} {\sl Existence and uniqueness theorems for systems of partial differential equations.} Proc. Symp. Fluid Dynamics and 
Appl. Math. (Univ. of Maryland, 1961), Gordon and Breach, New York (1962), 147-195.

\bibitem{Dr-Fo} {\sc Drinovec-Drnov\v sek, B. ; Forstneri\v c, F.} {\sl Hyperbolic domains in real Euclidean spaces.} ArXiv:2109.06943

\bibitem{Fo-Ka} {\sc Forstneri\v c,F. ; Kalaj, D.} {\sl Schwarz-Pick lemma for harmonic maps which are conformal at a point.} ArXiv:2102.12403


\bibitem{Gr}  {\sc Gromov, M.} {\sl Metric structures for Riemannian and non-Riemannian spaces.} Birkhauser Boston, Inc., MA, 2007, xx+585 pp.

\bibitem{Jo1} {\sc Jost, J.} Harmonic mappings between Riemannian manifolds.
Proceedings of the Centre for Mathematical Analysis, Australian National University, Vol. 4. Canberra: Centre for Mathematical Analysis, Australian National University. vi, 177 p. (1983). 

\bibitem{Jo2} {\sc Jost, J.}  Riemannian geometry and geometric analysis. Berlin: Springer-Verlag. xi, 401 p. Springer, 1995.

\bibitem{Hi} {\sc Hildebrandt, S.} {\sl Boundary behavior of minimal surfaces.} Arch. Rat. Mec. Anal. {\bf 35}(1969), 47-82. 

\bibitem{Ke} {\sc Kelley, J.L.} {\sl General topology}, volume 27 of Graduate Texts in Mathematics. Springer-Verlag, New York-Berlin, 1975. Reprint of the 1955 edition [Van Nostrand, Toronto, Ont.].

\bibitem{La} {\sc Lax, P.D.} {\sl  A stability theorem for solutions of abstract differential equations, and its application to the study of the local behavior of solutions of elliptic equations.}
Commun. Pure Appl. Math. 9, 747-766 (1956). 

\bibitem{Mo1}  {\sc Morrey, Ch. B.} Multiple integrals in the calculus of variations. Spinger Verlag,  1966.

\bibitem{Mo2} {\sc Morrey, Ch. B.} {\sl The problem of Plateau on a Riemannian manifold.} Ann. Math. {\bf 49}(1948), 807-851.

\bibitem{Ro} {\sc Royden, H.L.} {\sl Remarks on the Kobayashi metric.} Several Complex Variables II, Conf. Univ. Maryland 1970, 125-137 (1971). 
\bibitem{Wh1}  {\sc White,B.} {\sl The space of m-dimensional surface that are stationary for a parametric elliptic functional.} Indiana Univ. Math. J. {\bf 36} (1987), 567-602.
\bibitem{Wh3} {\sc White, B.} {\sl Generic regularity of unoriented two-dimensional area minimazing surfaces.} Ann. Math. {\bf 121} (1985), 595-603.



\end{thebibliography}
\end{document}